\documentclass[11pt]{amsart}
\usepackage{pstricks,pst-plot,}
\usepackage{color}
\usepackage{cancel}
\usepackage{tikz}
\usepackage{float}
\usepackage{enumitem}
\usepackage{cancel}

\definecolor{Black}{cmyk}{0,0,0,1}
\definecolor{OrangeRed}{cmyk}{0,0.6,1,0} 
\definecolor{DarkBlue}{cmyk}{1,1,0,0.20}
\definecolor{myblue}{rgb}{0.66,0.78,1.00}
\definecolor{Violet}{cmyk}{0.79,0.88,0,0}
\definecolor{Lavender}{cmyk}{0,0.48,0,0}

\newcommand\spiral{}
\def\spiral[#1](#2)(#3:#4:#5){
	\pgfmathsetmacro{\domain}{pi*#3/180+#4*2*pi}
	\draw [#1,shift={(#2)}, domain=0:\domain,variable=\t,smooth,samples=int(\domain/0.08)] plot ({\t r}: {#5*\t/\domain})
}
\parskip=\smallskipamount

\newtheorem{theorem}{Theorem}[section]
\newtheorem{lemma}[theorem]{Lemma}
\newtheorem{corollary}[theorem]{Corollary}
\newtheorem{proposition}[theorem]{Proposition}

\theoremstyle{definition}
\newtheorem{definition}[theorem]{Definition}

\newtheorem{remark}[theorem]{Remark}

\newcommand{\C}{\mathbb{C}}

\newcommand{\Z}{\mathbb{Z}}

\newcommand{\R}{\mathbb{R}}

\renewcommand{\Re}{\operatorname{Re}}

\newcommand{\bea}{\begin{eqnarray*}}
\newcommand{\eea}{\end{eqnarray*}}

\numberwithin{equation}{section}

\begin{document}
\title[Weighted $L^2$ version of Mergelyan and Carleman approximation]{Weighted $L^2$ version of Mergelyan and Carleman approximation}

\author{S\'everine Biard}
\author{John Erik Forn\ae ss}
\author{Jujie Wu}
\address{ {
$^{*}$ Corresponding author  \\
Jujie Wu} \
\\{\it E-mail address:}  jujie.wu@ntnu.no
\\{ School of Mathematics and statistics, Henan University}\\{ Jinming Campus of Henan University, Jinming District, City of Kaifeng, Henan Province. P. R. China, 475001}   \\ \ \ \ \  \ \ \ \ \ \ \ \ \ \ \ \ \ \ \ \ \ \ \ \ \ \ \ \ \ \ \  \ \ \ \ \ \ \  \ \ \ \ \ \ \  \ \ \ \ \ \ \   \ \ \    \ \ \ \ \ \ \ \ \ \  \ \ \ \ \ \ \ \ \ \  \ \ \ \ \ \ \ \ \ \ \ \ \ \ \ \ \ \ \ \ \ \ \ \ \ \ \ \ \ \ \ \  \ \ \ \ \ \ \ \ \ \ \ \ \ \ \ \ \ \ \  \ \ \ \ \ \ \ \ \ \ \ \ \ \ \ \ \ \ \ \ \ \ \ \ \ \ \ \ \ 
 \  \ \ \  \ \ \ \ \ \ \ \  \ \ \ \ \ \ \ \ \ \ \ \ \    \\{Department of Mathematical Sciences, NTNU}\\{ Sentralbygg 2, Alfred Getz vei 1, 7034 Trondheim, Norway} }

\address{{S\'everine Biard}
\\{\it E-mail address:} biard@hi.is
\\{Science Institute, University of Iceland}\\{Dunhagi 3, IS-107 Reykjavik, Iceland}\\
{\textit{current address}: Univ. Polytechnique Hauts-de-France,  - LMI - Laboratoire de Math\'ematiques pour l'Ing\'enieur, FR CNRS 2956, F-59313 Valenciennes, France}\\
\textit{current e-mail address}: severine.biard@uphf.fr}

\address{{John Erik Forn\ae ss}
\\{\it E-mail address:} john.fornass@ntnu.no
\\{ Department of Mathematical Sciences, NTNU}\\{ Sentralbygg 2, Alfred Getz vei 1, 7034 Trondheim, Norway}}


\date{}
\maketitle

\bigskip

\begin{abstract}
We study the density of polynomials in $H^2(E,\varphi)$, the space of square integrable functions with respect to $e^{-\varphi}dm$ and holomorphic on the interior of $E$ in $\mathbb{C}$, where $\varphi$ is a subharmonic function and $dm$ is a measure on $E$. We give a result where $E$ is the union of a Lipschitz graph and a Carath\'{e}odory domain, which we state as a weighted $L^2$-version of the Mergelyan theorem. 
We also prove a weighted $L^2$-version of the Carleman theorem.
\bigskip

\noindent{{\sc Mathematics Subject Classification} (2010):30D20, 30E10, 30H50, 31A05}

\smallskip

\noindent{{\sc Keywords}: Mergelyan theorem, Carleman theorem, Weighted $L^2$- spaces, Rectifiable non-Lipschitz arc} 
\end{abstract}
\tableofcontents

\section{Introduction}
Let $E \subset \C$ be a measurable set, $m$ a measure on $E$ and $\varphi$ a measurable function, locally bounded above on $E$. Denote by $L^2(E ,\varphi)$ the space of measurable functions $f$ in $E $ which are square integrable with respect to the measure $e^{-\varphi}dm$ i.e.,
$$L^2(E,\varphi):=\left \{f\mid \Vert f\Vert_{L^2(E,\varphi)}^2=\int_E\vert f\vert^2e^{-\varphi}dm<\infty \right \}.$$
Set 
$$
H^2(E, \varphi) = L^2(E, \varphi) \cap \mathcal{O}(\mathring{E})
$$ 
 where $\mathcal{O}(\mathring{E})$ stands for the space of holomorphic functions on the interior of $E$.

 In this paper, we generalize the classical holomorphic approximation theorems to weighted $L^2$-spaces. The theory of holomorphic approximation started in 1885 with two now classical theorems: the Weierstrass theorem and the Runge theorem. The first one states that a continuous function on a bounded interval of $\mathbb{R}$ can be approximated arbitrarily well by polynomials for the uniform convergence on the interval. We prove the following weighted $L^2$-version of the Weierstrass theorem: 
 \begin{theorem}\label{Weierstrass ab}
 Let $\gamma$ be a Lipschitz graph over a bounded interval and $\varphi$ a subharmonic function in a neighborhood of $\gamma$ in $\mathbb{C}$. Then polynomials are dense in $L^2(\gamma, \varphi)$.
\end{theorem}

 Recall that a Carath\'{e}odory domain $\Omega$ is a simply-connected bounded planar domain whose boundary $\partial{\Omega}$ is also the boundary of an unbounded domain. 
 Combined with Theorem 1.3 from \cite{BFW2019}, it leads us to the following weighted $L^2$-version of the Mergelyan theorem:

\begin{theorem}\label{MergelyanC}
	Let $\Omega \subset \mathbb{C}$ be a Carath\'{e}odory domain and $\gamma \subset \mathbb{C}$ a Lipschitz graph with one endpoint $p$ in $\overline{\Omega}$, the rest of $\gamma$ be in the unbounded component of the complement of $\overline{\Omega}$. Assume that the boundary of $\Omega$ is $\mathcal{C}^2$ near $p$. Let $\varphi$ be a subharmonic function in a neighborhood of $\overline{\Omega} \cup \gamma$. Then polynomials are dense in $H^2(\Omega\cup\gamma, \varphi)$. 
\end{theorem}
 Here, 
 \begin{eqnarray*}
 \begin{aligned} 
H^2(\Omega \cup \gamma, \varphi):=\displaystyle{\left\lbrace {f \ \ \text{is measurable on} \ \ \Omega \cup \gamma \ \text{with} \ f|_\Omega   \in\mathcal{O}(\Omega)} \right.} \\
	\left.{\ \text{and}  \ \int_\Omega\vert f(z)\vert^2e^{-\varphi(z)}d\lambda_z+ \int_\gamma \vert f\vert^2e^{-\varphi}ds <\infty }  
	\right\rbrace,
 \end{aligned}
\end{eqnarray*}
where $d\lambda_z$ is the Lebesgue measure on $\Omega$ and $ds$ is the arc length element.\\

 We can even generalize to 
	\begin{theorem}\label{MergelyanUC11}
	 We use the previous notations and assumptions. If $K_\ell$, $\ell=1,\cdots, N$ is either a Lipschitz graph or a bounded Carath\'{e}odory domain in $\mathbb{C}$ as in Theorem \ref{MergelyanC} such that $\mathbb{C}\setminus \left(\bigcup_{\ell=1}^N K_\ell\right)$ is connected, $\overline{K_i}$ and $\overline{K_j}$ have at most one common point $\lbrace{p_{ij}}\rbrace$ and $K_j$ is outside of the relatively compact connected component of $\overline{K_i}^c$ for each $i \neq j$. Let $\varphi$ be a subharmonic function in a neighborhood of $\bigcup_{\ell=1}^NK_\ell$. Then polynomials are dense in $H^2(\bigcup_{\ell=1}^NK_\ell, \varphi)$.
\end{theorem}

\begin{center}
	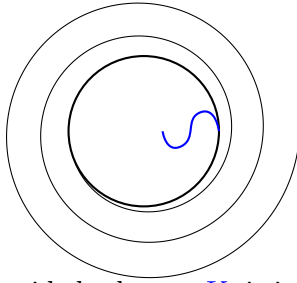
\begin{figure}[H]
		\begin{tikzpicture}
		\draw  (0,0) 
		\foreach \i [evaluate={\r=(\i/2000)^2;}] in {1000,...,2880}{ -- (\i:\r)}; 
		
		\path [fill=white] (0,0) circle (1);
		\draw[thick, black] (0,0) circle (1);
		\draw[thick, blue] plot [smooth,tension=1.5] coordinates{ (1, 0) (0.75,0.25) (0.5, -0.2) (0.25,0)};
		\end{tikzpicture}
		\caption{Case avoided: the arc $\color{blue}{K_j}$ is in the bounded component of $\overline{K_i}^c$ where $K_i$ is the outer snake}
	\end{figure}
\end{center}

 Let $\Gamma$ be the graph of a locally Lipschitz function over the real axis in $\C$. We may assume $\Gamma = \{(t, \phi(t)) \}$ with $\phi: \R \rightarrow \R$ a locally Lipschitz continuous function. 
 Thanks to Theorem \ref{Weierstrass ab}, we prove the following weighted $L^2$-version of the Carleman theorem:
 
 \begin{theorem} \label{L2Carleman}
 	Let $\Gamma$ be the graph of a locally Lipschitz function over the real axis in $\C$ and $\varphi$ a subharmonic function in a neighborhood of $\Gamma$. Denote by $\Gamma_n$, $n\in \mathbb{Z}$ the part of the graph $\Gamma$ over the interval $[n,n+1]$. Then for any $f\in L^2(\gamma, \varphi)$ and for any positive numbers $\varepsilon_n$, there exists an entire function $F$, so that for each $n\in \mathbb{Z}$,	
 	$$
 	\int _{\Gamma_n} |F -f|^2 e^{-\varphi} d s < \varepsilon_n.
 	$$ 
 \end{theorem}

The paper is organized as follows: In Section 2, we give a necessary and sufficient condition in terms of Lelong number for the exponential $e^{-\varphi}$ to be integrable on an arc in $\mathbb{C}$. In Section 3, we prove Theorem \ref{Weierstrass ab}. In Section 4, we prove Theorem \ref{MergelyanC}. In Section 5, we prove Theorem \ref{L2Carleman}.
In the last section, we give an example which shows that there are no non-zero polynomials in $L^2(\gamma, \varphi)$ for some rectifiable non-Lipschitz arc $\gamma$ and some subharmonic function $\varphi$.
\\

\section{``Exponential integrability" on arcs in $\C$}

\indent In the following, we assume that $\varphi$ is subharmonic on $\mathbb{C}$, even though it is enough to assume it to be subharmonic in a neighborhood of the given subset. This fact relies on the following lemmas.\\

\begin{lemma} \label{le:measurecomsubharmonic}
	Let $\mu \geq 0$ be a measure on $\C$ with finite mass on each compact set. Then there exists a subharmonic function $\varphi$ on $\C$ so that $\Delta \varphi = \mu$. 
\end{lemma}
\begin{proof}
	Set 
	$$
	\mu = \sum \limits_{n=1}^{\infty} \mu_n,  \ \ \mu _n = \mu_{\mid_{\{n-1 \leq |z| < n\} }}.
	$$
	 Let $\varphi_n (z)= \int \log |z-\zeta| d \mu_n(\zeta)$. Then $\varphi_n(z)$ is harmonic on $\{z: |z| < n-1\}$. Thus there exists a holomorphic function $h_n(z)$ on $\{z: |z| < n-1\}$ so that $\varphi_n(z) = \Re h_n(z)$. By using the classical Mergelyan theorem to $h_n(z)$ on $\{z: |z| \leq n-2\}$ there exists a polynomial $P_n(z)$ so that 
	 $$
	 |h_n(z)-  P_n(z)| < \frac{1}{2^n} ,  \ \ \text{on}\quad  \lbrace{  |z| \leq n-2}\rbrace.
	 $$
	 Then $\varphi (z)= \sum \limits_{n=1}^{\infty} \Re (  h_n(z) - P_n(z))$ is subharmonic on $\C$ and $\Delta \varphi = \mu$.
\end{proof}
\begin{lemma} \label{le:decomposition}
	Let $U \subset  \subset  V$ be two open sets and $\varphi$ a subharmonic function on $V$. Then there exists a subharmonic function $\psi$ on $\C$ so that $\varphi  = \psi + h$ on $U$, where $h$ is harmonic on $U$. 
\end{lemma}
\begin{proof}
	 Choose a smooth cut off function $\chi: \C \rightarrow [0,1]$ so that $\chi \equiv 1$ on a neighborhood of $\overline{U}$ and $\rm{supp}\chi \subset V$. 
    Then	$\mu:= (\Delta\varphi) \cdot \chi$ is a positive measure with finite mass on each compact set in $\mathbb C$. By Lemma \ref{le:measurecomsubharmonic} there is a globally defined subharmonic function $\psi$ such that $\Delta\psi=\mu.$ But then $\varphi=\psi+h$ on $U$ for some harmonic function $h$.
	\end{proof}

Since $h$ is uniformly bounded on $U$, as a direct consequence, we get 
\begin{lemma}
	Let $E\subset \subset U$ and $\varphi, \psi$ be as in Lemma \ref{le:decomposition}. Then the Hilbert spaces $L^2(E, \varphi) = L^2(E, \psi )$ and the norms are ``equivalent", i.e., there exist positive constants $C_1, C_2$ so that 
	
	\begin{eqnarray*}
		C_1 \|f \|_{L^2(E, \varphi)} \leq \| f \|_{ L^2(E, \psi)} \leq C_2 \|f \| _{L^2(E, \varphi)}.
	\end{eqnarray*}
	
\end{lemma}

We  prove similar statements to those in Section 2 of \cite{Fornaessandwu2017} but on an arc $\gamma$. Those results will allow us to prove at the end of the section the local integrability of the exponential $e^{-\varphi}$ at a point $x\in\gamma$ 
if and only if the Lelong number $\nu(\varphi)(x)$ is strictly less than 1.\\

\begin{definition}
	A function $f: E\rightarrow \R $, $E \subset \R$, is said to be $L-$Lipschitz, $L\geq 0$, if 
	$$
	|f(t_1) - f(t_2)| \leq L|t_1 -t_2|
	$$ 
	for every pair of points $(t_1,t_2) \in E\times E$. We say that a function is Lipschitz if it is $L-$Lipschitz for some $L$.
\end{definition}

\begin{lemma}[Chapter 5 of \cite{Royden1968}]
	If $f:[a,b] \rightarrow \R$ is a Lipschitz function, then $f$ is differentiable at almost every point in $[a,b]$ and 
	$$
	f(b)-f(a) = \int_a^b f'(t)dt.
	$$
\end{lemma}
Let $\gamma$ be the graph of a Lipschitz function $y(t): [a,b] \rightarrow \R$. In the following, we say that $\gamma$ is a \textit{Lipschitz graph} if $\gamma$ is the graph of an L-Lipschitz function $y$ on $[a,b]$, $a,b<\infty$. 
Locally Lipschitz graph means that for each point $p\in \gamma$, there exists a neighborhood $U$ where the graph is Lipschitz, up to a rotation of $\gamma$.
We denote by $|\gamma|$ the arc length of $\gamma$, defined as follows 
$$
|\gamma|: = \int_ \gamma ds = \int _a^b |\gamma'(t)|dt,
$$
where $|\gamma'(t)| = \sqrt{ 1+(y'(t))^2}$.

Let $z_0 = (t_0, y_0)\in \C$ and $0 < \beta < 1$.
\begin{lemma} \label{le:estimate for lip}
Let $\gamma$ be a graph, then $|\gamma(t) - z_0| \geq |t-t_0|$ on $[a,b]$. 
\end{lemma}

\begin{lemma}\label{lemmaestimateofgamma}
	Let $\gamma$ be a Lipschitz graph, then
	$$
	\int_a^b \frac{1}{|\gamma(t) - z_0|^\beta}|\gamma'(t)|dt \leq \rm{Const}_{L,a,b,\beta}.
	$$
\end{lemma}
\begin{proof}
By Lemma \ref{le:estimate for lip} we have
	\bea
	\int_a^b \frac{1}{|\gamma(t) - z_0|^\beta}|\gamma'(t)|dt	& \leq &   (L+1) \int_a^{b}  \frac{1}{|t- t_0|^\beta} dt. \nonumber \\
	\eea
	If $t_0>b$, then 
	\bea
	(L+1) \int_a^{b}  \frac{1}{|t- t_0|^\beta} dt & \leq & (L+1)  \int_a^{b}  \frac{1}{|t- b|^\beta}dt \nonumber \\
	& = & (L+1) \frac{(b-t)^{1-\beta}\mid^a_{b}}{1-\beta} \nonumber \\
	& = & (L+1) \frac{(b-a)^{(1-\beta)}}{1-\beta}; \nonumber \\
	\eea
	If $t_0 <a$, then
    \bea
   (L+1) \int_a^{b}  \frac{1}{|t- t_0|^\beta} dt & \leq & (L+1) \int_a^{b}  \frac{1}{|t- a|^\beta} dt \nonumber \\
    & = & (L+1) \frac{(t-a)^{1-\beta}\mid^b_{a}}{1-\beta}  \nonumber \\
    & = & (L+1) \frac{(b-a)^{(1-\beta)}}{1-\beta};  \nonumber \\
	\eea
	If $a \leq t_0 \leq b$, then 
	\bea
	(L+1) \int_a^{b}  \frac{1}{|t- t_0|^\beta}  dt
& \leq &  (L+1) \int_a^{t_0} \frac{1}{(t_0- t)^\beta} dt + (L+1) \int_{t_0}^b \frac{1}{(t- t_0)^\beta} dt \\
	& \leq &(L+1) \frac{(t_0-t)^{1-\beta}}{1-\beta}\mid ^a_{t_0} +(L+1) \frac{(t-t_0)^{1-\beta}}{1-\beta}\mid^b_{t_0} \\
	& \leq & 2(L+1) \frac{(b-a)^{(1-\beta)}}{1-\beta}.\\
	\eea
Thus no matter whatever the condition on $t_0$, we have that	
$$ \int_a^b \frac{1}{|\gamma(t) - z_0|^\beta}|\gamma'(t)|dt	 \leq \rm{Const}_{L,a,b, \beta}.$$
\end{proof}

We generalize the previous Lemma to a product 

\begin{lemma}\label{lemma2.2}
Let $\gamma$ be a Lipschitz graph. Suppose $z_i\in \C$, $\beta_i>0$, for $i=1,\cdots, m$ and $\sum_{i=1}^{m} \beta_i=\beta<1.$
Then $$ \int_a^b \prod_{i=1}^m \left(    \frac{1}{|\gamma(t)-z_i| } \right)^{\beta_i} |\gamma'(t)|dt < \rm{Const}_{L,a,b,\beta}.$$
\end{lemma}
\begin{proof}
According to Corollary 2.3 in \cite{Fornaessandwu2017} we know that 
$$ 
 \prod_{i=1}^m \left(    \frac{1}{|\gamma(t)-z_i| } \right)^{\beta_i}  \leq   \sum _{i=1}^{m}   \frac{\beta_i}{\beta} \left(    \frac{1}{|\gamma(t)-z_i| } \right)^\beta.  
$$
By Lemma \ref{lemmaestimateofgamma}, we finally have 
$$ \int_a^b \prod_{i=1}^m\left(    \frac{1}{|\gamma(t)-z_i| } \right)^{\beta_i} |\gamma'(t)|dt < \rm{Const}_{L,a,b,\beta}.$$
\end{proof}

Let $\gamma$ be a Lipschitz graph. We take an arc length parametrization of $\gamma$ that we denote by $s$.
 Let $h$ be a function on $\gamma$, define $$\int_{\gamma}hds: = \int_a^b h(\gamma(t)) |\gamma'(t)|dt.$$

From Lemma \ref{lemma2.2}, we are able to prove
\begin{theorem}\label{estimatefor phi}
Let $\gamma$ be a Lipschitz graph. Let $\mu$ be any nonnegative measure with total mass $\beta<1$ on an open set $U$ in $\C$ containing $\gamma$. 
If $\varphi(z)=\int \log |z-\zeta| d\mu(\zeta)$, then we have 
$$
\int_{\gamma} e^{-\varphi}ds < C_{L,\beta,a, b}
$$
where $C_{L,\beta,a, b} >0$ is a constant depending on $L,\beta,a, b.$  
\end{theorem}

\begin{proof} 
Define $\psi_n(z,\zeta)=\max\{\log|z-\zeta|,-n\}$ and 
$$
\varphi_n(z)=\int \psi_n(z,\zeta)d\mu(\zeta).
$$ 
Then $\varphi_n$ is continuous and $\varphi_n\searrow \varphi$ pointwise.
Hence $e^{-\varphi_n(z)}\nearrow e^{-\varphi(z)}$. 
Therefore, 
it is enough to show

$$
\int_{\gamma} e^{-\varphi_n}  ds \leq   C_{L,\beta,a, b} + \frac1n.
$$
We fix $n$. Let $\delta>0.$ 
Since $\psi_n$ is continuous,  by Lemma 2.4 in \cite{Fornaessandwu2017}, there exist $\zeta_i\in U$ such that the measure $\mu_n=\sum_{i=1}^m \beta_i\delta_{\zeta_i}$ has total mass $\beta<1$ and
$
|\tilde{\varphi}_n -\varphi_n|< \delta,
$
where $$\tilde{\varphi}_n(z):= \int \psi_n(z,\zeta)d\mu_n(\zeta) = \sum_{i=1}^m\beta_i\psi_n(z,\zeta_i)\geq \sum_{i=1}^m\beta_i\log\vert z-\zeta_i\vert.$$
Hence, we get
\begin{align}
\int_\gamma e^{-\varphi_n}ds & \leq e^\delta \int_{\gamma} e^{-\tilde{\varphi}_n}ds\nonumber\\
&\leq e^{\delta}  \int_a^b \prod_{i=1}^m\dfrac{1}{\vert \gamma(t) - \zeta_i\vert^{\beta_i}}|\gamma'(t)|dt.\label{phin}
\end{align}
By Lemma \ref{lemma2.2}, we get  
\begin{equation}\label{prod} 
\int_a^b \prod_{i=1}^m\dfrac{1}{\vert \gamma(t) - \zeta_i\vert^{\beta_i}}|\gamma'(t)|dt<  C_{L,\beta,a, b}.
\end{equation}
Hence, by combining \eqref{phin} and \eqref{prod} and by choosing $\delta$ small enough, we get 
$$\int_{\gamma}  e^{-\varphi_n}ds\leq  C_{L,\beta,a, b}+ \dfrac{1}{n}.$$
\end{proof}

\begin{corollary}\label{cor real}
Theorem \ref{estimatefor phi} holds with possibly larger constant for subharmonics function $\varphi$ on a neighborhood $U$ of $\gamma$ in $\C$ with $\mu:= \dfrac{1}{2\pi}{\Delta \varphi}_{\mid_U}$ being of total mass on $U$ strictly less than $1$.
\end{corollary}
\begin{proof}
By Riesz decomposition theorem (see for example Theorem 3.7.9 in \cite{Rans95}), we can decompose any subharmonic function as $\varphi(z)=\int_U \log\vert z-\zeta\vert d\mu(\zeta)+ h(z)$ where $h$ is harmonic.
Because $h$ is bounded on $U$, Theorem \ref{estimatefor phi} gives the result with a constant depending, in addition, on $h$. 
\end{proof}
We recall now the definition of the Lelong number $\nu(\varphi)$ of a subharmonic function $\varphi$ at a point $z$ in $\C$, where $\varphi\not\equiv -\infty$, that is equal to the mass of the ``Riesz" measure $\mu =\dfrac{1}{2\pi}\Delta \varphi$ at the point $z$. Below is an equivalent definition given by \cite{Ki1,Ki2},
$$\nu(\varphi)(z):=\lim_{r\to 0^+} \dfrac{\max_{\vert \zeta-z\vert =r} \varphi(\zeta)}{\log r}.$$
From this definition, Kiselman \cite{Ki1}, Theorem 4.1 proves that if $e^{-\varphi}$ is locally integrable at $z$ in $\C^n$ then $\nu(\varphi)(z)<2n$.

We recall the following result which is a converse of the previous property in complex dimension one:

\begin{theorem}\label{expintcomplex}
If $\varphi\not\equiv -\infty$ is subharmonic and $\nu(\varphi)(z)<2$ for a point $z,$ then $e^{-\varphi}$ is locally integrable in a neighborhood of $z$.
\end{theorem}

\begin{proof}
We refer to the note at top of p. 99 in H\"ormander's book \cite{Hormander66}. This is also a consequence of Theorem 2.5 in \cite{Fornaessandwu2017}. We also refer to Proposition 7.1 in \cite{Skoda1972} for a generalization to higher dimensions\footnote{ Many authors refer to this result as ``Skoda's exponential integrability"}.
\end{proof}
Now we state the previous Theorem for points in a Lipschitz graph.

\begin{theorem}\label{expintreal}
Let $\gamma$ be a Lipschitz graph and $\varphi\not\equiv -\infty$ a subharmonic function on $\C$. Then $\nu(\varphi)(\gamma(t))<1$ for a point $\gamma(t)\in \gamma$ if and only if $e^{-\varphi}$ is locally integrable in a neighborhood of $\gamma(t)$ on $\gamma$.
\end{theorem} 

\begin{proof}
$\Rightarrow)$ By hypothesis, $\nu(\varphi)(\gamma(t))=\dfrac{1}{2\pi}\Delta\varphi(\gamma(t))<1$ so this is a direct consequence of Corollary \ref{cor real}.

$\Leftarrow)$ If there exists $z_0 =(t_0,y_0)\in \gamma$ so that $\nu(\varphi)(z_0)\geq 1$, then we have $\frac{1}{2\pi} \Delta \varphi (z_0) \geq 1$. Therefore we have $\frac{1}{2\pi} \Delta \varphi - \delta_{z_0} \geq 0$ on $\C$. Hence we may find a subharmonic function $\psi$ on $\C$ such that $\frac{1}{2\pi}\Delta \psi = \frac{1}{2\pi}\Delta \varphi - \delta_{z_0}$.  
Up to a harmonic function we may write $\varphi  = \log |z-z_0| + \psi $. Thus

\begin{eqnarray*}
\int _{U(z_0) \cap \gamma} e^{-\varphi}ds& =&\int _{U(z_0) \cap \gamma} \frac{1}{|z-z_0|} e^{-\psi}ds, \psi \ \text{is locally bounded above near} \ z_0 \\ \nonumber 
	&  \geq & C \int _{U(z_0) \cap \gamma}\frac{1}{|z-z_0|} ds \\ \nonumber 
	&  \geq  & C \int_{t_0-\eta} ^{t_0+\eta} \frac{1}{\sqrt{1+L^2}|t-t_0|} dt   \\ \nonumber 
 &  = &	C \frac{1}{\sqrt{1+L^2}} \left(  \int_{t_0-\eta} ^{t_0} \frac{1}{t_0-t}dt +  \int_{t_0} ^{t_0+\eta} \frac{1}{t-t_0}dt \right)  \\ \nonumber 
 &  = &	C \frac{1}{\sqrt{1+L^2}} \left( \lim_{\varepsilon \rightarrow 0^+} \int_{t_0-\eta} ^{t_0-\varepsilon} \frac{1}{t_0-t}dt + \lim_{\varepsilon \rightarrow 0^+}  \int_{t_0+ \varepsilon} ^{t_0+\eta} \frac{1}{t-t_0}dt \right)  \\ \nonumber 
 &  = &	2C \frac{1}{\sqrt{1+L^2}}  \lim_{\varepsilon \rightarrow 0^+} (\ln \eta - \ln \varepsilon)  \\ \nonumber 
	&  = & \infty,
\end{eqnarray*}
where $\eta$ is a sufficiently small constant. 
\end{proof}
Let $\alpha>0$. We will write $\nu(\varphi)<\alpha$ to mean that $\nu(\varphi)(z)<\alpha$ for all points $z$ in the given subset. However, $\nu(\varphi)\geq \alpha$ should be understood as there exist at least one point $z$ such that $\nu(\varphi)(z)\geq \alpha$.\\

\section{Proof of Theorem \ref{Weierstrass ab}}

In this section, we prove Theorem \ref{Weierstrass ab} which generalizes the Weierstrass theorem to weighted $L^2$-spaces on Lipschitz graphs. The particularity here is to allow the weight $\varphi$ to have singularities on the given set. We will then carefully take the local integrability of $e^{-\varphi}$ (Section 2) into account.

We first recall some classical results

\begin{theorem}[Weierstrass 1883 \cite{Weierstrass1885}]\label{weierstrasstheorem}
Suppose $f$ is a continuous function on a closed bounded interval $[a, b] \subset \mathbb{R}$. For each $\varepsilon >0$ there exists a polynomial $P$ such that
$$
| f(x) - P(x) | < \varepsilon, \ \ \ \   \forall x \in [a,b].
$$
\end{theorem}

\begin{theorem}[Lavrent'ev, 1936 \cite{Laventi1936}]\label{Rungetheorem}
	Let $K\subset \C$ be compact with $\C \setminus K$ connected. Suppose
	that $f$ is continuous on $K$. If $\mathring{K} = \emptyset$, then for each $\varepsilon >0$ there exists a polynomial $P$ such that
	$$
	| f(x) - P(x) | < \varepsilon, \ \ \ \ \forall x \in K.
	$$
\end{theorem}

\begin{theorem}[Mergelyan, 1951 \cite{Mergelyan1951}]\label{Mergelyantheorem}
Let $K\subset \C$ be compact with $\C \setminus K$ connected. Suppose
 that $f$ is continuous on $K$ and holomorphic on $\mathring{K}$. Then, for each $\varepsilon >0$ there exists a polynomial $P$ such that
	$$
	| f(x) - P(x) | < \varepsilon, \ \ \ \ \forall x \in K.
	$$
\end{theorem}
 


In order to prove Theorem \ref{Weierstrass ab}, we first prove

\begin{theorem}\label{Weierstrassetensio}
Let $\gamma$ be a Lipschitz graph over $[a,b]$. Suppose that $\varphi$ is measurable on $\gamma$ and that
$$\int_{\gamma} e^{-\varphi}ds<\infty.$$
Then polynomials are dense in $L^2(\gamma, \varphi).$
\end{theorem}
In order to prove this theorem we need the following Lemma. 

For each nonnegative real valued $f \in L^2(\gamma, \varphi)$, define $f_n=\min \{f,n\}.$ 
\begin{lemma}\label{$L^2$approximation}
$f_n\rightarrow f$ in $L^2(\gamma, \varphi)$ as $n\nearrow \infty$.
\end{lemma}
\begin{proof}
We need to prove that $\int_{\gamma} |f-f_n|^2e^{-\varphi}ds\rightarrow 0.$ Now $|f|^2e^{-\varphi}$ is an $L^1$ function and  $|f-f_n|^2e^{-\varphi}\leq  |f|^2 e^{-\varphi}.$ Moreover
$|f-f_n|^2e^{-\varphi}\rightarrow 0$ a.e.. Hence by the Lebesgue dominated convergence theorem, we have that
$$\int_{\gamma} |f-f_n|^2e^{-\varphi}ds \rightarrow 0.$$
\end{proof}
\begin{proof} [Proof of Theorem \ref{Weierstrassetensio}]
Let $f$ be such that $\int_{\gamma} |f|^2e^{-\varphi} ds<\infty.$ To approximate $f$ by polynomials, it suffices to consider the case when $f\geq 0$. If $P$ is any polynomial, then $\|P-f\|_{L^2(\gamma, \varphi)}\leq \|P-f_n\|_{L^2(\gamma, \varphi)}+\|f_n-f\|_{L^2(\gamma, \varphi)}$. So by Lemma \ref{$L^2$approximation}, it suffices to approximate nonnegative bounded measurable functions $f_n$ by polynomials. 

Each bounded measurable function $f_n$ can be uniformly approximated by simple functions on $\gamma$. Since $\int_{\gamma} e^{-\varphi}ds<\infty$, this approximation also holds in the weighted norm of $L^2(\gamma, \varphi)$. 
Each simple function is the finite linear combination of characteristic functions on measurable sets in $\gamma$, hence it suffices to approximate a characteristic function by polynomials on measurable sets in $\gamma$. 
Since $e^{-\varphi}$ is $L^1$ integrable, there exists for any $\varepsilon>0$ a constant $\delta>0$
so that if $E$ is any measurable set in $\gamma$ with Lebesgue measure $|E|<\delta,$ then $\int_E e^{-\varphi}ds<\varepsilon.$
For any measurable set $F$ there exists a finite union of graphs over intervals $I$ so that
$|F\setminus I|,|I\setminus F|<\delta/2.$ Hence $\|\chi_F-\chi_I\|_{L^2(F, \varphi)}<\varepsilon.$
Then it suffices to approximate a characteristic function of a graph over an interval in $[a,b]$ in $\gamma$. 

Let $I$ be a graph over an interval in $\gamma$ and $\delta>0$. There exists  $f\in\mathcal{C}^0(I,\mathbb{C})$ so that 
$$
B:=\lbrace{s\in \C \mid f(s)\neq \chi_{I}(s)}\rbrace  \ \ \  \text{has a Lebesgue measure less than} \ \ \  \frac \delta 8.
$$
So
$$\int_I |f-\chi_I|^2 e^{-\varphi}ds \leq \int_{B} e^{-\varphi}ds \leq  \varepsilon.$$

By the Lavrent'ev theorem, it suffices to approximate a continuous function by polynomials, and then we are done.
\end{proof}

\begin{remark}
Theorem \ref{Weierstrassetensio} holds for non-Lipschitz rectifiable graphs.
\end{remark}

\indent Theorem \ref{Weierstrass ab} generalizes Theorem \ref{Weierstrassetensio} by relaxing the assumption on the integrability of $e^{-\varphi}$ on $\gamma$. To prove this generalization, we need the following result.

\begin{theorem}\label{th:gene1}
Let $\gamma$ be a Lipschitz graph over $[a,b]$ and $\varphi$ a subharmonic function on $\C$. Then polynomials are dense in $L^2(\gamma, \varphi)$ if and only if the function $\sqrt{Q}$ can be approximated arbitrarily well by polynomials in $L^2(\gamma, \varphi)$ where $Q$ is a polynomial vanishing at the points $\gamma(t)$ to order $[\nu(\varphi)(\gamma(t))]$ with Lelong number $\nu(\varphi)(\gamma(t)) \geq 1$. 
\end{theorem}
Here $[\nu(\varphi)(\gamma(t))]:=\max\lbrace{m\in\mathbb{Z}\mid  \nu(\varphi)(\gamma(t)) \geq m}\rbrace$ is also called the floor function of $\nu(\varphi)$ at $\gamma(t)$.
\begin{proof} 
By Theorem \ref{expintreal} and Theorem \ref{Weierstrassetensio}, it suffices to consider $\nu(\varphi) \geq 1$. Then there exist finitely many points $\gamma(t_1), \cdots, \gamma(t_n)$ such that $\nu(\varphi)(\gamma(t_i))\geq 1$ for $t_i\in[a,b]$, $i=1,\dots, n$ and $Q$ which can be expressed as $Q(z) = \prod_{j=1}^n (z-\gamma(t_i))^{[\nu(\varphi)(\gamma(t_i))]}$. We may then choose a subharmonic function $\psi$ so that $\varphi = \psi + \log |Q|$ with $\nu(\psi)<1$ on $\gamma$. \\
$\Rightarrow)$ Remark that $\sqrt{Q}\in L^2(\gamma, \varphi)$ by Theorem \ref{expintreal}:
$$
\int_\gamma\vert\sqrt{Q}\vert^2e^{-\varphi}ds=\int_\gamma \vert\sqrt{Q}\vert^2 e^{-\psi-\log\vert Q\vert}ds=\int_\gamma e^{-\psi}ds<\infty.
$$
Then by assumption, there exists a polynomial that approximates arbitrarily well $\sqrt{Q}$.\\
$\Leftarrow)$ 
Remark that $f\in L^2(\gamma, \varphi)$ is equivalent to
$$
\int_\gamma \left| \frac{f}{\sqrt{Q}}\right|^2 e^{-\psi} d s  < \infty.
$$
So $\dfrac{f}{\sqrt{Q}}\in L^2(\gamma, \psi)$ and
by Theorem \ref{Weierstrassetensio}, for each $\varepsilon>0$, there exists a polynomial $P$ so that  
$$
\int_\gamma \left| \frac {f}{ \sqrt{Q}} - P \right|^2 e^{-\psi} ds = \int_\gamma |f - P\sqrt{Q}|^2 e^{-\varphi} ds< \varepsilon. 
$$
Thus if $\sqrt{Q}$ can be approximated arbitrarily well by polynomials in $L^2(\gamma, \varphi)$, then $f$ can be approximated by polynomials in $L^2(\gamma, \varphi)$.	
\end{proof}	
\begin{proof} [Proof of Theorem \ref{Weierstrass ab}]
	By Theorem \ref{expintreal} again, it suffices to consider $\nu(\varphi) \geq 1$. Let $Q$ and $\psi$ be as in the proof of Theorem \ref{th:gene1}. 
	We prove that for every $\varepsilon>0$, there is a 
	polynomial $P$ that vanishes at $\gamma(t_i)$ to order $[\nu(\varphi)(\gamma(t_i))]$ so that  
	\begin{equation}\label{estimate111}
	\int_\gamma \left|\sqrt{Q} -P\right|^2 e^{-\varphi}ds =  \int_\gamma \left | 1-  \frac{P}{\sqrt{Q} }\right |^2 e^{-\psi} ds  < \varepsilon.
	\end{equation}
	For convenience, we look for $P$ such as
	$$P(z) = Q(z)  \cdot \widetilde{P}(z) = \prod_{j=1}^n (z-\gamma(t_j))^{[\nu(\varphi)(\gamma(t_j))]} \cdot \widetilde{P}(z).$$
	Then $(\ref{estimate111})$ is equivalent to find some polynomial $\widetilde{P}$ so that 
	\begin{eqnarray}\label{estimate2}
	\int_\gamma \left|1 -  \frac{Q \cdot \widetilde{P}}{\sqrt{Q}}\right|^2 e^{-\psi} ds < \varepsilon.
	\end{eqnarray}
	Let $\delta>0$. Set
	$g(z)= \dfrac{1}{\sqrt{Q(z)}}$ except on the arcs $I_i$ on $\gamma$ with length $2\delta$ and center at $\gamma(t_i)$. We can make $g$ continuous and $\vert \sqrt{Q(z)}g(z)\vert\leq 1$ on such arcs of length $2\delta$.
	Then 
\begin{eqnarray}\label{eq31}
 \int_\gamma \left |1 - \sqrt{Q}g\right\vert^2e^{-\psi}ds & \leq & 4\sum_{i=1}^n\int_{I_i}e^{-\psi}ds\nonumber\\
 & = & 4\int_{\bigcup_{i=1}^n I_i} e^{-\psi}ds.\label{xj}
 \end{eqnarray}
 Since $\bigcup_{i=1}^n I_i$ is a measurable set of measure $2\delta n$ and $e^{-\psi}\in L^1_{\text{loc}}$, we may choose $\delta$ sufficiently small in order for \eqref{xj} to be $<\varepsilon$.
Since $g\in L^2(\gamma, \psi)$, by Theorem \ref{Weierstrassetensio}, there exists a polynomial $A$ satisfying 
	$$
	\int_\gamma|g- A|^2 e^{-\psi} ds < \frac {\varepsilon}{\left(\max\limits_{\gamma(t)\in  \gamma} \vert Q(\gamma(t))\vert\right)^2}. 
	$$
	Then by Cauchy-Schwarz and the previous estimate, 
	\begin{eqnarray}\label{eq412121}
	&& \int_\gamma \vert \sqrt{Q}\vert^2\vert g-A\vert^2 e^{-\psi}ds \\  \nonumber
	&&\leq \left(\int_\gamma \vert \sqrt{Q}\vert^4 \vert g-A\vert^2e^{-\psi}ds\right)^{1/2}\left(\int_\gamma\vert g-A\vert^2e^{-\psi}ds\right)^{1/2},\\ \nonumber
	&&\leq  \left(\max_{\gamma}\vert Q\vert\right)^2 \cdot \int_\gamma \vert g-A\vert^2e^{-\psi}ds,\\ \nonumber
	&&< \varepsilon.
	\end{eqnarray}
Combine \eqref{eq31} and \eqref{eq412121}, we may choose $\widetilde{P} = A$.
\end{proof}	

\begin{remark}
In fact we show that $\mathcal{P} \cap L^2(\gamma, \varphi)$ is dense in $L^2(\gamma, \varphi)$ where $\mathcal{P}$ is the set of all polynomials. If $\nu(\varphi) \geq 1$, not all polynomials are in $L^2(\gamma, \varphi)$, e.g. $1\notin L^2(\gamma, \varphi)$.
\end{remark}

\section{Proof of Theorem \ref{MergelyanC}}
\indent In this section, we give a weighted $L^2$- version of the Mergelyan theorem for compact sets in $\mathbb{C}$ that are the union of bounded Carath\'{e}odory domains and Lipschitz graphs. 
Recall the following theorem\\
\begin{theorem}[Theorem 1.3 in \cite{BFW2019}] \label{th:unit disc}
Let $\Omega \subset \mathbb{C}$ be a Carath\'{e}odory domain and $\varphi$ a subharmonic function in $\mathbb{C}$. Then polynomials are dense in $H^2 (\Omega, \varphi)$.
\end{theorem}
One generalization of Theorem \ref{th:unit disc} is the following 
	\begin{theorem}\label{pr:twocaratheodory}
		Let $\Omega_1 \subset \mathbb{C}$ be a Carath\'{e}odory domain, let $\Omega_2$ be another Carath\'{e}odory domain which is inside of a bounded component of $\overline{\Omega_1}^c$. Let $\varphi$ be a subharmonic function in $\mathbb{C}$. Then polynomials are dense in $H^2(\Omega_1\cup\Omega_2, \varphi)$.
	\end{theorem}
\begin{center}
	\begin{figure}[H]
		\begin{tikzpicture}
		\draw [blue, domain=0:50,variable=\t,smooth,samples=500]
		plot ({\t r}: {2+2*exp(-0.1*\t)});
		\draw  (0,0) 
		\foreach \j [evaluate={\s=(\j/2000)^2;}] in {1000,...,1800}{ -- (\j:\s)}; 
		\path [fill=white] (0,0) circle (0.33);
		\draw[thick, red] (0,0) circle (0.33);
		\draw[thick, red] (0,0) circle (2);
		\draw[line width=6pt, white] (0,0) circle (0.90);
		\draw (0.5,0.5) node[right] {$\Omega_2$};
		\draw[blue] (2.5,2) node[right] {$\Omega_1$};
		\end{tikzpicture}
		\caption{Example of a situation that might happen in Theorem \ref{pr:twocaratheodory} where $\Omega_1$ and $\Omega_2$ are outer snakes with no common point.}
	\end{figure}
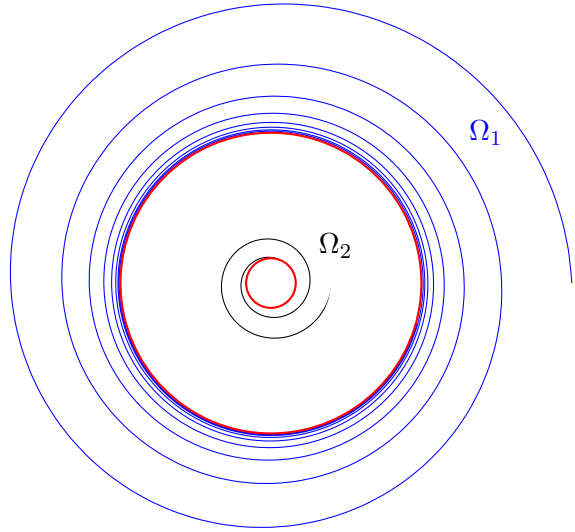
\end{center}

In order to prove this theorem we need the following proposition which can be easily deduced from the proof of  Proposition 1.2 in \cite{BFW2019}.
\begin{proposition} \label{le:caratheodory}
	Let $\Omega$ be a bounded Carath\'{e}odory domain and $\varphi$ a subharmonic function in $\C$. Then for each $f\in H^2(\Omega, \varphi)$ there exist functions $f_n \in H^2(\Omega_n, \varphi)$ such that  $\|f_n -f \|_{L^2( \Omega, \varphi)} \rightarrow 0$ and $\|f_n\|_{L^2(\Omega_n \setminus \Omega, \varphi)} \rightarrow 0$ as $n\rightarrow \infty $, where $\Omega_n\supset \overline{\Omega}$ is a
	sequence of bounded simply-connected domains so that $\partial \Omega_n$ converges to $\partial\Omega$ in the sense of the Hausdorff distance. 
\end{proposition}
\begin{proof}[Proof of Theorem \ref{pr:twocaratheodory}]
First we consider the case
		\begin{eqnarray}
		\nu (\varphi) < 2  \ \ \ \ \ \text{on} \ \ \overline{\Omega}_1\cup\overline{\Omega}_2. 
		\end{eqnarray}		
By Theorem \ref{th:unit disc}, for each $\varepsilon >0$ there exists a holomorphic polynomial $P_1$ so that 

\begin{eqnarray}\label{ineq:caraunioncara12}
\int_{\Omega_2} | f-P_1|^2e^{-\varphi} d\lambda_z < \frac {\varepsilon} {26}.
\end{eqnarray}
Put 
$$
g=\left\{
\begin{aligned}
f-P_1  & & \text{on}  \ \ \Omega_1\\
0 & &  \ \  \text{on}  \ \ \ \Omega_2.
\end{aligned}
\right.
$$
Since $\Omega_i, i =1,2$ is a Carath\'{e}odory domain, there exists a sequence $\{\Omega_{i,n}\}$ of Jordan domains such that $\overline{\Omega_i}\subset \Omega_{i,n}$ and ${\overline{\Omega}_{i,n+1}}\subset \Omega_{i,n}$ and the Hausdorff distance between $\partial \Omega_{i,n}$ and $\partial\Omega_i$ tends to zero as $n\rightarrow \infty$  (see Chapter I, Section 3 of \cite{Gaier80}). Let $n$ be sufficiently large so that $\Omega_1 \cap \Omega_{2,n} = \emptyset$. \\
Since $g\in H^2(\Omega_1, \varphi)$, by using Proposition \ref{le:caratheodory} on $\Omega_1$, for each $\varepsilon>0$ we get for large enough $n$ functions $g_{1,n} \in H^2{(\Omega_{1,n}, \varphi)}$ so that 
\begin{eqnarray}\label{eq:1}
\|g_{1,n} -g \|_{L^2( \Omega_1, \varphi)} < \frac{\varepsilon }{26}
\end{eqnarray}
and 
\begin{eqnarray}\label{eq:2}
 \|g_{1,n}\|_{L^2(\Omega_{1,n} \setminus \Omega_1, \varphi)} < \frac{\varepsilon }{26}.
\end{eqnarray}
By Theorem \ref{th:unit disc}, for each $n$ there exists a polynomial $Q_n$ so that 
\begin{eqnarray}\label{eq:3}
\int_{\Omega_{1,n}} |g_{1,n}-Q_n|^2 e^{-\varphi}d \lambda < \frac{\varepsilon }{26}.
\end{eqnarray}
Thus for sufficiently large $n$ we have
\begin{eqnarray*}\label{convergence1}
	& &\int_{\Omega_1} |f- P_1-Q_n|^2 e^{-\varphi}d\lambda_z \nonumber \\
	&& \leq 2 \int_{\Omega_1} |f- P_1-g_{1,n}|^2 e^{-\varphi}d\lambda_z  + 2\int_{\Omega_1} |g_{1,n}-Q_n|^2 e^{-\varphi}d\lambda_z   \nonumber \\
	&& \leq  2\int_{\Omega_1} |g - g_{1,n} |^2 e^{-\varphi}d\lambda_z +  \frac{\varepsilon}{13}  \nonumber  \ \ \ \text{ by (\ref{eq:3})}\\
	&& \leq   \frac{2\varepsilon}{13}  \nonumber \ \ \ \text{ by( \ref{eq:1})} ;\\
\end{eqnarray*}
and
\begin{eqnarray*}\label{convergence2}
& &\int_{\Omega_2} |f - P_1 - Q_n|^2 e^{-\varphi}d\lambda_z \nonumber \\
&& \leq 2 \int_{\Omega_2} |f- P_1|^2 e^{-\varphi}d\lambda_z  + 2\int_{\Omega_2} |Q_n|^2 e^{-\varphi}d\lambda_z   \nonumber \\
&& \leq  \frac{2\varepsilon}{26} + 4 \int_{\Omega_2} |Q_n - g_{1,n}|^2 e^{-\varphi}d\lambda_z + 4 \int_{\Omega_2} |g_{1,n}|^2 e^{-\varphi}d\lambda_z 
 \nonumber \ \ \ \text{ by (\ref{ineq:caraunioncara12})}  \\
&& \leq \frac{5\varepsilon}{13} \ \ \ \  \text{by (\ref{eq:2}}), (\ref{eq:3}). \nonumber \\ 
\end{eqnarray*}
Next we consider the case $\nu(\varphi) \geq 2$. Then 
there exist finitely many points $z_j \in \overline{\Omega}_1\cup\overline{\Omega}_2$, $1 \leq j \leq N$ so that $\nu (\varphi) \geq 2$ at those points. We may find a polynomial $Q$ with zeros at those points, subharmonic function $\psi$ satisfying $\varphi=\psi + 2\log \vert Q \vert$ and $\nu (\psi) <2$ on $ \overline{\Omega}_1\cup\overline{\Omega}_2$. Let $f\in \mathcal{O}(\Omega_1\cup\Omega_2) \cup L^2 (\Omega_1\cup\Omega_2, \varphi)$, $$\frac{f}{Q} \in \mathcal{O}(\Omega_1\cup\Omega_2) \cap L^2(\Omega_1\cup\Omega_2, \psi).$$
Then by the first case, $\frac{f}{Q}$ can be approximated by a polynomial $P$ in 
$L^2(\Omega_1\cup\Omega_2, \psi)$.
Thus $f$ can be approximated by the polynomial $P \cdot Q$ in $L^2 (\Omega_1\cup\Omega_2, \varphi)$.
\end{proof}
The proof of the main theorem of this section, Theorem 1.2, is divided into 3 cases that correspond to the locus of the zeros of the polynomial $Q$ in the decomposition of the weight function $\varphi=\psi+\log \vert Q\vert$. We will then need the following result:
\begin{theorem}\label{th:simplyconnectedpeak}
	Let $\Omega$ be a Carath\'{e}odory domain and $\varphi$ a subharmonic function in $\mathbb{C}$. Suppose $P$ is a polynomial with $\int _\Omega |P|^2 e^{-\varphi} d \lambda < \infty$. Let $p\in \partial \Omega$ and a disc $\Delta \subset \Omega^c$ with $p\in \partial \Delta$. Then we can approximate $P$ by $\widetilde{P}$ which is holomorphic on a neighborhood of $\overline{\Omega}$ with $\widetilde{P}(p) = 0$ in the norm of $L^2(\Omega, \varphi)$. 
\end{theorem} 
\begin{proof}
 Set $ M= \int _\Omega |P(z)|^2 e^{-\varphi(z)} d \lambda_z$. Then for each $\varepsilon >0$ there exists a small neighborhood $U(p) $ of $p$ so that 
	$$
	\int_{U(p)\cap \Omega} |P(z)|^2 e^{-\varphi(z)} d\lambda_z  < \frac{\varepsilon}{2}.
	$$
Take $h_n(z) = \left (\frac{p-q}{z-q} \right)^n$, choose $n$ sufficiently large so that 
	\begin{eqnarray*}
	|h_n(z)|^2 < \frac{\varepsilon}{2M} \ \text{on } \ \left(U(p)\right)^c \cap \Omega \ \ \text{and}  \ |h_n(z)|^2<1  \ \text{on}  \ U(p) \cap \Omega.
	\end{eqnarray*}
	Now fix such $n$. Set $\widetilde{P}(z) =(1-h_n(z)) \cdot P(z)$, then $\widetilde{P}$ is holomorphic on a neighborhood of $\Omega$ with $\widetilde{P}(p) = 0$ satisfying 
	\begin{eqnarray*}
		\int_{\Omega} |\widetilde{P}(z) -P(z)|^ 2e^{-\varphi(z)} d\lambda_z
		& =  &\int_{U(p) \cap \Omega} |h_n(z)P(z)|^ 2e^{-\varphi(z)} d\lambda_z + \\&&\int_{\Omega \setminus U(p)} |h_n(z)P(z)|^ 2e^{-\varphi(z)} d\lambda_z \nonumber \\
		& \leq &\max_{z \in U(p)\cap \Omega}|h_n(z)|^2 \cdot \int_{U(p)\cap \Omega} |P(z)|^2 e^{-\varphi(z)} d\lambda_z  \nonumber \\ 
		&& + \max_{z \in \left(U(p)\right)^c \cap \Omega}|h_n(z)|^2 \cdot \int_{\left(U(p)\right)^c \cap \Omega} |P(z)|^2 e^{-\varphi(z)} d\lambda_z  \nonumber \\
		& <& \frac{\varepsilon}{2} +  \frac{\varepsilon}{2}\nonumber \\
		& = & \varepsilon.\nonumber \\
	\end{eqnarray*}
\end{proof}
\begin{remark}
The assumption on the existence of such a disc on the boundary of $\Omega$ is verified when $\partial\Omega$ is $\mathcal{C}^2$.
\end{remark}

\begin{proof} [Proof of Theorem \ref{MergelyanC}]
	First we consider: Case 1:  
	\begin{eqnarray*}
		\nu(\varphi)<1 \ \  \text{on}  \ \  \gamma \ \ \  \text{and} 
		\ \ \  \nu(\varphi) < 2 \ \   \text{on}    \ \  \overline{\Omega}.
	\end{eqnarray*}
	Let $f\in \mathcal{O}(\Omega) \cup L^2 (\Omega \cup \gamma, \varphi)$. By Theorem \ref{th:unit disc}, for each $\varepsilon >0$, there exists polynomial $P_1$ so that                             
\begin{eqnarray} \label{eq:estimateonomega}
	\int _{\Omega} |f(z) - P_1(z)|^2 e^{-\varphi(z)}d\lambda_z < \frac {\varepsilon} {16}. 
\end{eqnarray}
	By the Theorem \ref{Weierstrass ab}, there exists polynomial $P_2$ so that 
	\begin{eqnarray}\label{eq:estimateongamma}
		\int_\gamma |f- P_2| ^2 e^{-\varphi} ds <\frac {\varepsilon} {16}.
	\end{eqnarray}
 Since $\Omega$ is a Carath\'{e}odory domain, there exists a sequence $\{\Omega_j\}$ of bounded simply-connected domains such that $\overline{\Omega}\subset \Omega_j$ and ${\overline{\Omega}_{j+1}}\subset \Omega_j$ and the Hausdorff distance between $\partial \Omega_j$ and $\partial\Omega$ tends to zero as $j\rightarrow \infty$.
By Corollary \ref{cor real}	we may choose $j$ sufficiently large so that 
\begin{eqnarray} \label{caratheodorysmall}
	\max \left \{ \int_{\gamma \cap \Omega_j} |f|^2 e^{-\varphi}ds,  \ \ \ \  \int_{\gamma \cap \Omega_j} \left| P_1\right |^2 e^{-\varphi}ds \right \} < \frac {\varepsilon} {64}. 
\end{eqnarray}
	Now fix such $j$. 
	Let $\chi : \mathbb{C} \rightarrow [0,1]$ be a smooth function with $ \chi \equiv 1 $ on $ \Omega_{j+1}$ and $\chi \equiv 0$ outside of $\Omega_{j}$.
	Set 
	$$
	h(z) =  \chi (z) P_1(z) + (1-\chi(z)) P_2(z).
	$$ 
	Then $h(z)$ is holomorphic on $\Omega$, continuous on $\overline{\Omega} \cup \gamma$. Set $$M = \int_ \Omega e^{-\varphi(z)} d\lambda _z + \int_\gamma e^{-\varphi} ds.$$
	By Mergelyan approximation theorem, there exists a polynomial $P$ so that 
	$$
	|P -h|^2 < \frac {\varepsilon} {16M} \ \ \ \  \ \ \text{on}  \ \ \overline{\Omega} \cup \gamma.
	$$
	Then 
	\begin{eqnarray*}\label{eq41}
	&&\| f-P\|^2_{L^2 (\Omega \cup \gamma, \varphi)}\nonumber\\
	&\leq & 2 \| f-h\|^2_{L^2 (\Omega \cup \gamma, \varphi)} +2 \| h-P\|^2_{L^2 (\Omega \cup \gamma, \varphi)} \nonumber\\
	&  \leq &2  \int _{\Omega}| f(z) - P_1(z) | ^2e ^ {-\varphi (z)} d \lambda_z + 2  \int _{\left(\Omega_j\right)^c \cap \gamma} | f - P_2 | ^2e ^ {-\varphi } d s \nonumber\\
	&& + 2 \int_{\gamma \cap \Omega_j} |f- h|^2e ^ {-\varphi } d s + \frac {\varepsilon}{8} \nonumber\\
	& \leq& 2 \int_{\gamma \cap \Omega_j} |f- \chi \cdot P_1 - (1-\chi)P_2|^2e ^ {-\varphi} d s + \frac {3\varepsilon}{8}\quad\text{by (\ref{eq:estimateonomega}), (\ref{eq:estimateongamma})} \nonumber\\ 
	& \leq &4 \int_{\gamma \cap \Omega_j} |f- P_1|^2\chi^2e ^ {-\varphi} d s \nonumber\\
	&& +  4 \int_{\gamma \cap \Omega_j} |f- P_2|^2(1-\chi)^2e ^ {-\varphi} ds  +  \frac {3\varepsilon}{8} \nonumber\\
	& \leq &8 \int_{\gamma \cap \Omega_j} 
	|f|^2 e^{-\varphi}ds + 8 \int_{\gamma \cap \Omega_j} |P_1|^2 e^{-\varphi}ds + \frac{5\varepsilon}{8} \quad\text{by (\ref{eq:estimateongamma})}\nonumber\\
	& \leq &\frac{7\varepsilon}{8} \quad\text{by (\ref{caratheodorysmall})}. \nonumber\\
	\end{eqnarray*}
	Now we consider Case 2: 
	\begin{eqnarray*}
		1 \leq \nu(\varphi)< 2 \ \  \text{on}  \ \ \gamma \ \ \  \text{and} 
		\ \ \  \nu(\varphi)(z)< 2 \ \ \text{on}  \ \    \overline{\Omega}.
	\end{eqnarray*}
	Then there exist finitely many points $\gamma(t_i) \in \gamma$, $t_i\in [a,b]$,  $1\leq i\leq N$ such that
	\begin{eqnarray*}
		t_1<t_2 <t_3 < \cdots < t_N \quad\text{and}\quad 1 \leq \nu(\varphi)(\gamma(t_i)) < 2,
	\end{eqnarray*}
	a polynomial $Q$ and a subharmonic function $\psi$ satisfying $\varphi=\psi+\log \vert Q\vert$ where $Q$ vanishes only at $\gamma(t_i)$ and $\nu(\psi)<1$ on $\gamma$. Since the polynomial $Q$ has no zeros on $\overline{\Omega}\setminus \{\gamma(a)\}$, $\nu(\psi)(z)= \nu(\varphi)(z)$ for each $z\in \overline{\Omega}\setminus\{\gamma(a)\}$.\\ We need now to distinguish two subcases depending on the nature of $p:=\gamma(a)$:\\
	 Subcase A: If $t_1 =a$.
	Let $f\in \mathcal{O}(\Omega) \cup L^2 (\Omega \cup \gamma, \varphi)$. Then by Theorem \ref{th:unit disc} there exists a polynomial $P_1$ satisfying 
\begin{eqnarray*} 
	\left \| f - P_1 \right \|_{L^2(\Omega, \varphi )} < \varepsilon.
\end{eqnarray*}
	Since the boundary of $\Omega$ is $\mathcal{C}^2$ near $p$, by Theorem \ref{th:simplyconnectedpeak}, we may choose a holomorphic function $\mathcal{P}_1$ on $\Omega_j$ satisfying
	\begin{eqnarray}\label{eq:withoutzeropoint} 
	\mathcal{P}_1(p)=0  \ \ \ \text{and}  \ \ \ \  \left \| f -  \mathcal{P}_1 \right \|_{L^2(\Omega, \varphi )} < \frac{\varepsilon}{16}.
	\end{eqnarray}
	By the construction in the proof of Theorem \ref{Weierstrass ab}, there exists a polynomial $P_2$ so that 
	\begin{eqnarray}\label{ineq:gamma}
	\int_\gamma |f- P_2\cdot Q| ^2 e^{-\varphi} ds <\frac {\varepsilon} {16}.
\end{eqnarray}
	Set $\displaystyle{M=\int_{\Omega} e^{-\psi(z)}d\lambda_z +  \int_\gamma e^{-\psi}ds.}$ 
	Since $f\in \mathcal{O}(\Omega) \cup L^2 (\Omega \cup \gamma, \varphi)$, by (\ref{eq:withoutzeropoint}) we may choose $j$ sufficiently large so that $\gamma(t_i) \in (\Omega_j)^c$, $2\leq i \leq N$ and 
	\begin{eqnarray} \label{caratheodorysmall2}
	\max \left \{ \int_{\gamma \cap \Omega_j} |f|^2 e^{-\varphi}ds,  \ \ \ \  \int_{\gamma \cap \Omega_j} \left| \mathcal{P}_1\right |^2 e^{-\varphi}ds \right \} < \frac {\varepsilon} {64}. 
	\end{eqnarray}
	Now fix such $j$. 
	Choose $\chi$ be as above and
	$$
	h(z) =  \chi (z) \mathcal{ P}_1(z) + (1-\chi(z)) P_2(z) Q(z).
	$$ 
	Then $h$ and $\frac{h}{Q}$ are holomorphic on $\Omega_j$, continuous on $\overline{\Omega_j} \cup \gamma$. By Theorem \ref{Mergelyantheorem} there exists a polynomial $G$ so that 
	$$
	\left |\frac hQ-G \right |^2 < \frac {\varepsilon} {32 M \cdot  \max\limits_{\overline{\Omega_j} \cup \gamma }|Q|} \ \ \text{on }  \ \ \ \overline{\Omega} \cup \gamma .
	$$
	Then 
	\begin{eqnarray*}\label{eq:Omegaunionarc}
	&&\| f-G \cdot Q\|^{2}_{L^2 (\Omega \cup \gamma, \varphi)} \nonumber\\
	&  = & \left \| \frac{f}{\sqrt {Q}}-G \cdot \sqrt{Q} \right \|^{2}_{L^2 (\Omega \cup \gamma, \psi)}  \nonumber\\
	& =&   \left \| \frac{f}{\sqrt {Q}}- \frac{h}{\sqrt {Q}} + \frac{h}{\sqrt {Q}} - G \cdot \sqrt{Q} \right \|^{2}_{L^2 (\Omega \cup \gamma, \psi)}  \nonumber\\
	& \leq &2 \left \| \frac{f}{\sqrt {Q}}- \frac{h}{\sqrt {Q}}  \right \|^{2}_{L^2 (\Omega \cup \gamma, \psi)}   + 2 \left \|\frac{h}{\sqrt {Q}} - G \cdot \sqrt{Q} \right \|^{2}_{L^2 (\Omega \cup \gamma, \psi)}  \nonumber\\
	& \leq& 2 \| f-h\|^2 _{L^2 (\Omega \cup \gamma, \varphi)} + 2M \cdot \max_{\overline{\Omega} \cup \gamma } |Q| \cdot \max _ {\overline{\Omega} \cup \gamma } \left| \frac h Q  -G\right |^2 \nonumber\\
	&  \leq& 2  \int _{\Omega}\left | f(z)- \mathcal{P}_1(z) \right | ^2e ^ {-\varphi (z)} d \lambda_z  + 2  \int _{\gamma \cap (\Omega_{j})^c} \left | f -P_2\cdot  Q \right | ^2e ^ {-\varphi } d s  \nonumber\\
	&&  + 2  \int_{\gamma \cap (\Omega_{j} \setminus \Omega)} \left | f - h \right | ^2e ^ {-\varphi} d s + \frac {\varepsilon}{16} \nonumber\\
	& \leq& 2 \int_{\gamma \cap \Omega_{j}} \left|f- \chi  \mathcal{P}_1  - (1-\chi)P_2 \cdot Q \right|^2e ^ {-\varphi } ds + \frac {5\varepsilon}{16} \quad \text{by (\ref{eq:withoutzeropoint}), (\ref{ineq:gamma})} \nonumber\\
	& \leq& 4 \int_{\gamma \cap \Omega_{j} }  \left|f -\mathcal{P}_1 \right|^2\chi^2e ^ {-\varphi} ds + 4 \int_{\gamma \cap \Omega_{j} }  \left |f- P_2\cdot Q\right|^2(1-\chi)^2e ^ {-\varphi} d s  +  \frac {5\varepsilon}{16} \nonumber\\
	& \leq& 8 \int_{\gamma \cap \Omega_{j}}  
	|f|^2 e^{-\varphi}ds + 8 \int_{\gamma \cap \Omega_{j} }  \left|\mathcal{P}_1\right|^2 e^{-\varphi}ds + \frac{9\varepsilon}{16} \quad \text {by (\ref{ineq:gamma})}\nonumber\\
	& \leq & \frac{13\varepsilon}{16} \quad  \text{by (\ref{caratheodorysmall2})}.\nonumber\\
	\end{eqnarray*}
Subcase B: If $\gamma(t_1)\neq p$. We may choose $j$ sufficiently large so that $\gamma(t_i) \in (\Omega_j)^c \cap \gamma, 1\leq i \leq N$ and the formula (\ref{caratheodorysmall2}) also holds. Then the following proof is similar to subcase A.\\
	
	Finally we consider:  Case 3: $\nu (\varphi)\geq 2$. There exist finitely many points $z_j \in \overline{\Omega}$, $\gamma(t_j) \in \gamma$ so that $\nu (\varphi)\geq 2$ at those points. We may find a polynomial $Q_1$ with zeros at those points, subharmonic function $\psi$ satisfying $\varphi=\psi+ 2\log \vert Q_1 \vert$ and $\nu (\psi) <2$ on $ \overline{\Omega} \cup \gamma$. Let $f\in \mathcal{O}(\Omega) \cup L^2 (\Omega \cup \gamma, \varphi)$, $$\frac{f}{Q_1} \in \mathcal{O}(\Omega) \cap L^2(\Omega \cup \gamma, \psi).$$
	Then by Case 1 or 2, $\frac{f}{Q_1}$ can be approximated by a polynomial $P$ in $L^2(\Omega \cup \gamma, \psi)$.
	Thus $f$ can be approximated by the polynomial $P \cdot Q_1$ in $L^2 (\Omega \cup \gamma, \varphi)$.
\end{proof}

\section{Proof of Theorem  \ref{L2Carleman}}

The classical Carleman approximation theorem applies to continuous function $f$ on $\mathbb R.$ Let $\varepsilon(x) >0$ be a continuous function.
\begin{theorem} [\cite{Carleman}]
	
	There exists an entire function $F$ so that $|F(x)-f(x)|<\varepsilon(x)$ on $\R$.
\end{theorem}
This theorem is equivalent to the following corollary
\begin{corollary}
	Let $f$ be a continuous function on $\mathbb R$. For any $\{ \varepsilon_n\}_{n=-\infty}^{\infty} $ with $\varepsilon_n >0$, there exists an entire function $F$ so that for each $n$,
	\begin{eqnarray*}
		|F(x)-f(x)| < \varepsilon_n,  \ \ \ \ \ \forall x \in [n, n+1].
	\end{eqnarray*}
\end{corollary}
It was pointed out by Alexander \cite{Alexander} that Carleman's proof actually gives 
\begin{theorem}[\cite{Carleman, Alexander}]\label{th:Carleman, Alexander}
	If $\gamma: \R \rightarrow \C$, $\gamma$ is a locally rectifiable curve and properly embedded, then for each continuous function $f$ on $\gamma$ and continuous function $\varepsilon>0$, there exists an entire function $F$ so that $|F-f|<\varepsilon$ on $\gamma$.
	\end{theorem}

We prove below Theorem \ref{L2Carleman} which is a weighted $L^2$- version of this generalization for Lipschitz graphs.

Let $\Gamma$ be the graph of a locally Lipschitz function over the real axis in $\C$. We may assume $\Gamma = \{(t, \phi(t)) \}$ with $\phi: \R \rightarrow \R$ a locally Lipschitz continuous function. 
For each $[n, n+1]$, $\Gamma_n: =\{ (t, \phi(t))|\; n\leq t \leq n+1 \}$ is a Lipschitz graph. 
\begin{proof}[Proof of Theorem \ref{L2Carleman}]
	 Case 1: $\nu(\varphi)< 1$ on $\Gamma$. Let $f \in L^2(\Gamma, \varphi)$. By Theorem \ref{Weierstrassetensio}, 
	there exists a continuous function $g_n$ on $J_n: = \{(t,\phi(t)) | \; n-1 \leq t \leq  n+2 \}$ so that 
	$$
	\int_{J_n} |f-g_n|^2e^{-\varphi}ds < \frac{1}{40} \min \{\varepsilon_{n-1}, \varepsilon_{n}, \varepsilon_{n+1} \} \leq \frac{1}{40} \varepsilon_{n}.
	$$	
Choose a partition of unity $\{ \chi _n \}_{n\in \Z}$ of $\Gamma$ so that $\chi_n \geq 0$ on $\Gamma$, $\chi_n =0$ outside of $J_n$ and $\sum_n\chi_n =1 $ on $\Gamma$. Let $g = \sum _n \chi_n g_n$. Then $g$ is continuous on $\Gamma$ and for each $n$,	
	\begin{eqnarray}
		&& \int_{\Gamma_n} |f-g|^2 e^{-\varphi}ds \nonumber \\
		& = & \int_{\Gamma_n} | \chi_{n-1}(f-g_{n-1}) + \chi_{n}(f-g_{n}) + \chi_{n+1}(f-g_{n+1})|^2 e^{-\varphi} ds    \nonumber \\
		& \leq  & 2 \int_{\Gamma_n} | \chi_{n-1}(f-g_{n-1}) + \chi_{n}(f-g_{n})|^2 e^{-\varphi} ds  + 2 \int_{\Gamma_n}  |\chi_{n+1}(f-g_{n+1})|^2 e^{-\varphi} ds    \nonumber \\
		& \leq  & 4\int_{\Gamma_n} | \chi_{n-1}(f-g_{n-1})|^2 e^{-\varphi} ds  + 4 \int_{\Gamma_n} | \chi_{n}(f-g_{n})|^2 e^{-\varphi} ds \nonumber \\ 
		 & & + 2 \int_{\Gamma_n} | \chi_{n+1}(f-g_{n+1})|^2 e^{-\varphi} ds    \nonumber \\
		& \leq &  \frac{1}{4}  \varepsilon_n.  \nonumber \\
	\end{eqnarray}	
	By using Theorem \ref{th:Carleman, Alexander} on the continuous function $g$ of $\Gamma$, we can find an entire function $F$ so that for each $n$
	\begin{eqnarray*}
		|F- g|^2 < \frac{\varepsilon_n}{4 \int_{\Gamma_n} e^{-\varphi}ds} \ \ \ \text{on} \ \  \Gamma_n.
	\end{eqnarray*}
Thus 
$$
 \int_{\Gamma_n}  |g - F|^2e^{-\varphi} ds < \frac {\varepsilon_n}{4}  .
$$
Hence we have 
\begin{eqnarray*}
	& &	\int_{\Gamma_n}|f - F|^2e^{-\varphi} ds  \nonumber \\
	& &	\leq  2 \int _{\Gamma_n}|f - g|^2e^{-\varphi} ds  + 2	\int _{\Gamma_n}|g - F|^2e^{-\varphi} ds   \nonumber \\
	& &	< \varepsilon_n. \nonumber \\
\end{eqnarray*}

Now we consider: Case 2: $\nu(\varphi) \geq 1$. We may list 
the points $\{\Gamma(t_j)\}_j$ with $\nu(\varphi)(\Gamma(t_j)) \geq 1$, where $\Gamma(t_j) = (t_j, \phi(t_j))$. Then 
there exists an entire function $Q$ which vanishes at each $\Gamma(t_j)$ to exact order $[\nu (\varphi)(\Gamma(t_j))]$. 
We may define $\sqrt{Q}$ to be continuous on $\Gamma$, without loss of generality we may set $Q(z)= \prod\limits_j (z-\Gamma(t_j))^{[\nu(\varphi)(\Gamma(t_j))]} e^{p_j(z)}$, where $p_j(z)$ are entire functions. Then there exists a subharmonic function $\psi$ such that $\varphi = \psi + \log |Q|$ with $\nu(\psi) < 1$ on $\Gamma$. Let $f \in L^2(\Gamma, \varphi)$. Then 
$
\frac{f}{\sqrt{Q}} \in L^2( \Gamma, \psi).
$
By Case 1, there exists an entire function $F$ so that for each $n$
\begin{equation}\label{transforforf1}
\int_{\Gamma_n}\left |\frac{f}{\sqrt{Q}} -F    \right|^2e^{-\psi} ds = \int_{\Gamma_n} |f -F \cdot \sqrt{Q}|^2e^{-\varphi} ds < \frac{\varepsilon_n}{4}.
\end{equation}
Thus it suffices to find an entire function $H$ vanishing at $\Gamma(t_j)$ to order $[\nu(\varphi)(\Gamma(t_j))]$ so that for each $n$
\begin{equation} \label{eq:findH}
\int _{\Gamma_n} \left|\sqrt{Q} -H\right|^2 e^{-\varphi}ds =  \int _{\Gamma_n} \left | 1-  \frac{H}{\sqrt{Q} }\right |^2 e^{-\psi} ds  < \frac{\varepsilon_n}{4 \max\limits_{\Gamma(t) \in \Gamma_n}|F(\Gamma(t))|^2}.
\end{equation}
We look for
$H$ for convenience as
$$H(z) = Q(z)  \cdot \widetilde{H}(z) = \prod_{j=1}^\infty (z-\Gamma(t_j))^{[\nu(\varphi)(\Gamma(t_j))]}e^{p_j(z)} \cdot \widetilde{H}(z).$$
The estimate \eqref{eq:findH} is then equivalent to 
find an entire function  $\widetilde{H}$ so that 
\begin{eqnarray}\label{eq:findhilte}
\int _{\Gamma_n} \left|1 -  \frac{Q \cdot \widetilde{H}}{\sqrt{Q}}\right|^2 e^{-\psi} ds < \frac{\varepsilon_n}{4 \max\limits_{\Gamma(t) \in \Gamma_n}|F(\Gamma(t))|^2},  \ \ \ \ \forall \ n.
\end{eqnarray}
Let $\delta_j>0$. Set 
$g(z)= \dfrac{1}{\sqrt{Q(z)}}$ except on arcs $\Gamma^j$ of $\Gamma$ with length $2\delta_j$ and center at $\Gamma(t_j)$. We can make $g$ continuous and 
$\vert \sqrt{Q} \cdot g\vert\leq 1$ on such arcs of length $2\delta_j$. 
Then 
\begin{eqnarray*}\label{eq3122}
\int _{\Gamma_n} \left |1 - \sqrt{Q} \cdot g\right\vert^2e^{-\psi}ds & \leq & 4 \sum_{\Gamma(t_j) \in J_n}\int_{\Gamma^j}e^{-\psi}ds.\nonumber\\
\end{eqnarray*}
Since $\bigcup\limits_{\Gamma(t_j) \in J_n} \Gamma^j$ is a measurable set and $e^{-\psi}\in L^1_{\text{loc}}$, we may choose $\delta_j$ sufficiently small in order to 
$$
\sum_{\Gamma(t_j) \in J_n}  \int_{\Gamma^j} e^{-\psi}ds < \frac{\varepsilon_n}{32 \max\limits_{\Gamma(t) \in \Gamma_n}|F(\Gamma(t))|^2}.
$$\\
Since $g$ is continuous on $\Gamma$, by the classical Carleman approximation theorem there exists an entire function $A$ satisfying for each $n$ 
$$
|g - A|^2  \leq \frac{\varepsilon_n}{8 \max\limits_{\Gamma(t) \in \Gamma_n} |Q(\Gamma(t))| \cdot \max\limits_{\Gamma(t) \in \Gamma_n}|F(\Gamma(t))|^2 \cdot \int_{\Gamma_n} e^{-\psi} ds}, \ \ \ \forall \Gamma(t) \in \Gamma_n.
$$
Then by Cauchy-Schwarz and the previous estimate, for each $n$, 
\begin{eqnarray*}
	&& \int_{\Gamma_n} \vert \sqrt{Q}\vert^2\vert g-A\vert^2 e^{-\psi}ds \nonumber \\ 
	&&= \int_{\Gamma_n}  \vert \sqrt{Q} \vert^2 \vert g-A\vert^2e^{-\psi}ds  \nonumber \\
	&& \leq  \left ( \int_{\Gamma_n}  \vert \sqrt{Q} \vert^4 \vert g-A\vert^2e^{-\psi}ds \right)^{\frac12} \cdot \left ( \int_{\Gamma_n} \vert g-A\vert^2e^{-\psi}ds \right)^{\frac12}\nonumber \\
	&& \leq   \max _{\Gamma(t) \in \Gamma_n} |Q(\Gamma(t))| \cdot \int_{\Gamma_n}\vert g-A\vert^2e^{-\psi}ds \nonumber \\
	&&\leq  \frac{\varepsilon_n}{8\max\limits_{\Gamma(t) \in \Gamma_n}|F(\Gamma(t))|^2} .\nonumber \\
\end{eqnarray*}
By taking $\widetilde{H} =A$, we get (\ref{eq:findhilte}) and then (\ref{eq:findH}). Hence there exist entire functions $F, Q, A$ so that for each $n$
\begin{eqnarray*}
	&& \int_{\Gamma_n} \vert f-F\cdot Q\cdot A\vert^2 e^{-\varphi}ds \nonumber \\ 
	& \leq & 2 \int_{\Gamma_n} \vert f-F\sqrt{Q}\vert^2 e^{-\varphi}ds + 2 \int_{\Gamma_n} \vert F\sqrt{Q} - F\cdot Q\cdot A\vert^2 e^{-\varphi}ds  \nonumber \\
	& \leq & \frac{\varepsilon_n}{2} + \frac{\varepsilon_n}{2}  \nonumber \\
	& = & \varepsilon_n.\nonumber \\
\end{eqnarray*}
\end{proof}

\section{Rectifiable non-Lipschitz arcs}
Here we construct a rectifiable non-Lipschitz arc $\gamma$ and a subharmonic function $\varphi$ in a neighborhood of $\gamma$ so that the conclusion of Theorem \ref{Weierstrass ab} does not hold. To find such an arc, we first look at the vertical arcs
$\gamma_a=\{(a,it),|t|\leq a\}$. Let $0<\alpha <1, z_0=0.$ 
We then notice that

\bea
|\gamma_a| = 2 a  \sim  a,
\eea
and
\bea
\int_{\gamma_a}\frac{ds}{|z-z_0|^\alpha} &  =  &  \int _{-a}^a \frac{1}{\left(\sqrt{a^2 + t^2}\right)^\alpha}dt. \nonumber \\
\eea
Since 
\begin{eqnarray*}
 \frac{2}{\sqrt{2}^\alpha} a^{1-\alpha} \leq  \int _{-a}^a \frac{1}{\left(\sqrt{a^2 + t^2}\right)^\alpha}dt  \leq  2a^{1-\alpha} \nonumber \\ 
\end{eqnarray*}
we get that $\int_{\gamma_a}\frac{ds}{|z-z_0|^\alpha}  \sim  a^{1-\alpha}$ uniformly in $\alpha.$\\

Let $b_n\in [0,1], n=1,2,3,\cdots, $ be a decreasing sequence which tends to 0. We will fix $b_n$ later. We remark that by setting $\varphi(z) = \sum\limits_{n=1}^\infty \alpha_n \log \left| \frac{z-b_n}{2} \right|$, where $\alpha_n=\frac{1}{n^3}$ a rapidly then $\varphi$ is subharmonic on $\Delta(0,2-b_1)$:  $\varphi$ is the limit of the decreasing sequence $\{ \varphi_k\}$ of subharmonic functions, $\varphi_k = \sum \limits_{n=1} ^k \alpha_n \log \left|\frac{z-b_n}{2} \right|$. 
Now let's build a rectifiable non-Lipschitz arc $\gamma$ such that
\begin{enumerate}
	\item \ \ \ $| \gamma |  < \infty$;	
	\item  \ \ \ 	$\int_{\gamma_n} e^{-\varphi}ds = \infty$. 
\end{enumerate} 
Here $\gamma_n$ is a curve with endpoints $(b_{n+1},0)$ and $(b_n,0)$ and $\gamma$
consists of the union of the $\gamma_n$ and the origin.
Define $c_n$ so that $$
c_n^{\frac{1}{1-\alpha_{n+1}}} = \frac{1}{n^2} \frac{\alpha_{n+1}}{1-\alpha_{n+1}} = \frac{1}{n^2 \left( (n+1)^3 -1 \right)}.
$$
We define $\{b_n\}$ by the following conditions: 
 $$
 b_n - b_{n+1} = c_n^{\frac{1}{1-\alpha_{n+1}}} =\frac{1}{n^2 \left( (n+1)^3 -1 \right)}
 $$
\noindent add the requirement that $b_n \rightarrow 0.$
Then we have that 
$$
b_n = \sum\limits_{k\geq n}^\infty (b_k - b_{k+1}) = \sum\limits_{k\geq n}^\infty \frac{1}{k^2 \left( (k+1)^3 -1 \right)}.
$$

Now define $b_n^k\in [ b_{n+1}, b_n]$ satisfying 
\begin{eqnarray} \label{bnk}
	b_n^k & = & b_{n+1} + \left(\frac{c_n}{k}\right)^{ \frac{1}{1-\alpha_{n+1}}} \nonumber \\
	& = & b_{n+1} + \frac{1}{n^2 \left( (n+1)^3 -1 \right)} \left(\frac{1}{k}\right)^{ \frac{1}{1-\alpha_{n+1}}}
\end{eqnarray}
Then $b_n^k\rightarrow b_{n+1}$ as $k\rightarrow \infty$ and $b_n^1=b_n.$
Finally, we define $\gamma_n = \left( \cup_{k=1}^\infty \gamma_n^k \right) \cup T_n\cup S_n  $, where 
$$
\gamma_n^k = : \{b_n^k + i y, 0 \leq y \leq b_n^k - b_{n+1} \},
$$
 $T_n= \cup_{\ell=0}^\infty \{y_x-b_{n+1}, b_n^{2\ell+2}\leq x\leq b_n^{2\ell+1}\}$ 
and $S_n=\cup_{\ell\geq 1}^\infty \{y=0, b_n^{2\ell+1}\leq x\leq b_n^{2\ell}\}$ which connect the $\gamma_n^k$ making $\gamma_n$ an arc. 
Then  $|\gamma_n^k| =  (b_n^k-b_{n+1})$, $|T_n| < \sqrt{2}(b_n-b_{n+1})=\sqrt{2}(b_n^1-b_{n+1})$
and $|S_n|<b_n-b_{n+1}$. Then, we have 
\begin{eqnarray*}
	|\gamma| &= &\sum\limits_{n=1}^\infty |\gamma_n|  \\
	&=&\sum\limits_{n=1}^\infty \sum\limits_{k=1}^\infty|\gamma_n^k| + \sum\limits_{n=1}^\infty |T_n| + \sum\limits_{n=1}^\infty |S_n|\nonumber \\
	& \leq & \sum\limits_{n=1}^\infty \sum\limits_{k=1}^\infty  (b_n^k-b_{n+1}) + \sum\limits_{n=1}^\infty (\sqrt{2}
+1)
(b_n^1-b_{n+1})\nonumber \\ 
	&=& \sum\limits_{n=1}^\infty \sum\limits_{k=1}^\infty \frac{1}{n^2 \left( (n+1)^3 -1 \right)} \left(\frac{1}{k}\right)^{ \frac{1}{1-\alpha_{n+1}}} + (\sqrt{2}+1) \sum\limits_{n=1}^\infty \frac{1}{n^2 \left( (n+1)^3 -1 \right)}\nonumber \\ 
\end{eqnarray*}
Since 
\begin{eqnarray*}
\sum\limits_{k=1}^\infty \left(\frac{1}{k}\right)^{ \frac{1}{1-\alpha_{n+1}}} \sim \int_1^\infty \left(\frac{1}{x}\right) ^{ \frac{(n+1)^3}{(n+1)^3-1}}dx= (n+1)^3-1
\end{eqnarray*}
we know that 
\begin{eqnarray}
	|\gamma|  \leq  C\sum\limits_{n=1}^\infty \frac{1}{n^2} + (\sqrt{2}+1) \sum\limits_{n=1}^\infty \frac{1}{n^2 \left( (n+1)^3 -1 \right)} <\infty. 
\end{eqnarray}
Thus $\gamma$ is a rectifiable non-Lipschitz arc.\\
On the other hand, we have
\begin{eqnarray} \label {gamman}
	\int_{\gamma_n} e^{-\varphi}ds & \geq &\sum\limits_{k=1} ^ \infty \int_{\gamma_n^k} e^{-\varphi}ds \nonumber \\
	&\geq & \sum\limits_{k=1} ^ \infty  \int_{\gamma_n^k} \frac{1}{|z-b_{n+1}|^{\alpha_{n+1}}}ds \nonumber \\
	& \sim & \sum\limits_{k=1} ^ \infty  (b_n^k - b_{n+1})^{1-\alpha_{n+1}} \nonumber \\
	& \sim & \sum\limits_{k=1} ^ \infty  \frac{c_n}{k}   \ \ \ \ \ \text{by (\ref{bnk})} \nonumber \\
	& \sim & \int_1^\infty \frac{c_n}{x} dx \nonumber \\
	& = & \infty, \text{ $\forall n$}.
\end{eqnarray}
Now we will prove that polynomials are not dense in $L^2(\gamma, \varphi)$. By contradiction, for each $f\in L^2(\gamma, \varphi)$, if there exists a sequence of polynomials $P_N$ so that 
$$
\int_{\gamma} |f-P_N|e^{-\varphi} ds \rightarrow 0,  \ \ \text{if} \ \ N\rightarrow \infty,
$$
then by (\ref{gamman}) we have $P_N(b_n) = 0$ for any $n$ if $N$ is sufficiently large. Since $b_n \rightarrow 0$ by uniqueness property of holomorphic function we know $P_N\equiv 0$. Thus $\int_{\gamma}|f|^2e^{-\varphi}ds =0$. That is $f =0$ a.e. on $\gamma$. Thus $L^2(\gamma, \varphi) = \{0\}$. On the other hand $f(z) := e^{\frac{\varphi}{2}} =\sqrt{\Pi \left|\frac{z-b_n}{2} \right|^{\alpha_n}} \in L^2(\gamma, \varphi)$. This is a contradiction. 
\begin{remark}
In this example, there are no non-zero polynomials in $L^2(\gamma, \varphi)$ and polynomials are not dense in $L^2(\gamma, \varphi)$. The key to this example is that Theorem \ref{expintreal} does not hold for rectifiable non-Lipschitz arcs. However, we don't know if there exists a rectifiable non-Lipschitz arc $\gamma$ and a subharmonic function $\varphi$ so that all the polynomials are in $L^2(\gamma, \varphi)$ but not dense in it.\\
\end{remark}

\textbf{Acknowledgements} The first author was supported in part by Rannis-grant 152572-051. The second author was supported by the Norwegian Research Council grant 240569. The third author was supported by the NSFC grant 11601120 and Norwegian Research Council grant 240569.

\end{document}